\documentclass[reqno]{amsart}
\usepackage{amsmath}
\usepackage{amssymb}
\usepackage{amsthm}
\usepackage{cases}
\usepackage{mathtools}
\mathtoolsset{showonlyrefs=true}
\usepackage{mathrsfs}
\usepackage{hyperref}

\title{A Fibonacci variant of the Rogers-Ramanujan identities via crystal energy}

\author{Shunsuke Tsuchioka}
\address{Department of Mathematical and Computing Sciences, Tokyo Institute of Technology, Tokyo 152-8551, Japan}
\email{tshun@kurims.kyoto-u.ac.jp}

\date{Sep 24, 2024}

\newtheorem{Thm}{Theorem}[section]

\newtheorem{Prop}[Thm]{Proposition}
\newtheorem{Lem}[Thm]{Lemma}

\newcommand{\REP}[1]{(#1)^{\ast}}
\DeclareMathOperator{\CL}{\mathsf{cl}}

\DeclareMathOperator{\STR}{\mathsf{Str}}

\newcommand{\EF}{\tilde{f}}
\newcommand{\GI}{\tilde{g}}
\DeclareMathOperator{\HT}{\mathsf{ht}}
\DeclareMathOperator{\WT}{\mathsf{wt}}
\DeclareMathOperator{\AF}{\mathsf{af}}

\newcommand{\BI}[1]{{#1}}

\newcommand{\LP}{\mathcal{P}}
\newcommand{\PB}{\mathbb{B}}

\makeatletter
\newcommand{\twoarrows}[3][0.2ex]{%
  \mathrel{\mathpalette\twoarrows@{{#1}{#2}{#3}}}%
}
\newcommand{\twoarrows@}[2]{\twoarrows@@#1#2}
\newcommand{\twoarrows@@}[4]{%
  \vcenter{\offinterlineskip\m@th
    \ialign{\hfil##\hfil\cr
      $#1#3$\cr
      \noalign{\vskip#2}
      $#1#4$\cr
    }%
  }%
}
\makeatother
\usepackage{graphicx}

\begin{document}
\maketitle

\begin{abstract}
We define a length function for a perfect crystal.
As an application, we derive a variant of the Rogers-Ramanujan identities, which involves (a $q$-analog of) the Fibonacci numbers.
\end{abstract}

\section{Introduction}
This paper is a continuation of ~\cite{Ts1}, where we gave a proof of the second Rogers-Ramanujan identity
via Kashiwara crystals.
The idea is summarized as follows.
\begin{quotation}
For an explicit realization $B\cong B(\lambda)$ of a highest weight $A$-crystal, find a ``length function'' $\ell:B\to\mathbb{Z}$
so that the generating function $F(x,q)=\sum_{b\in B}x^{\ell(b)}q^{|b|}$ behaves ``nicely''. 
\end{quotation}
Here, $|b|=n$ if $b=\widetilde{f}_{i_n}\cdots\widetilde{f}_{i_1}\emptyset$ for the highest weight element $\emptyset$ of $B$.

In ~\cite{Ts1},
for $A=A^{(1)}_1$ and $\lambda=3\Lambda_0$,
we adopt a connected component in the triple tensor product $B(\Lambda_0)^{\otimes 3}$ as a realization of $B(3\Lambda_0)$, where the basic crystal $B(\Lambda_0)$ is realized as the set of strict partitions $\STR$~\cite{MM}.
The value $\ell(x\otimes y\otimes z)$ of the function $\ell$ for $x\otimes y\otimes z$, where $x,y,z\in\STR$, is defined to be the sum
of the lengths of $x$, $y$ and $z$.
Then, we have (see ~\cite[Corollary 4.3]{Ts1})
\begin{align}
F(x,q)= (-xq;q)_{\infty}\sum_{s\geq 0}\frac{q^{s(s+1)}x^{2s}}{(q;q)_s},
\label{sankou}
\end{align}
where we use the usual convention for the $q$-Pochhammer symbols (see ~\cite{Ts1}).

The aim of this paper is to point out that, for a Kyoto path realization of a highest weight crystal~\cite{KKMMNN2,KKMMNN},
one can define a function $\ell_H$ (see \eqref{defH}), which we call the $H$-length, so that $F(x,q)$ satisfies a non-trivial $q$-difference
equation (Proposition \ref{mainres3}).
By applying it to $A^{(1)}_1$ Kirillov-Reshetikhin perfect crystal $B^{1,3}$ with a slight modification, we get a
variant of the Rogers-Ramanujan identities below.
\begin{Thm}\label{mainthm}
For $i=1,2$, we have
\begin{align*}
\sum_{n\geq 0}\frac{b^{(i)}_n}{(q;q)_n}=\frac{1}{(q^i,q^{5-i};q^5)_{\infty}},
\end{align*}
where the numerators are defined by $b^{(i)}_{n+2}=q^{n+2}b^{(i)}_{n}-q^{n+1}b^{(i)}_{n+1}$ for $n\geq 0$, and
$b^{(i)}_0=1$, $b^{(i)}_1=q$ (resp. $b^{(i)}_1=0$) for $i=1$ (resp. $i=2$).
\end{Thm}

One easily sees that $b^{(i)}_n$ is a sign coherent polynomial of $q$ and $(-1)^{n+i} b^{(i)}_n(1)$ is a Fibonacci number for $n>2$.
For example, we have 
\begin{align*}
b^{(1)}_2 &= 0,\quad
b^{(1)}_3=q^4,\quad
b^{(1)}_4=-q^7,\quad
b^{(1)}_5=q^9(1+q^2),\\
b^{(1)}_6 &= -q^{13}(1+q+q^3),\quad
b^{(1)}_7=q^{16}(1+q^2+q^3+q^4+q^6),\\
b^{(2)}_2 &= q^2,\quad
b^{(2)}_3=-q^4,\quad
b^{(2)}_4=q^6(1+q),\quad
b^{(2)}_5=-q^9(1+q+q^2),\\
b^{(2)}_6 &= q^{12}(1+q+q^2+q^3+q^4),\quad
b^{(2)}_7=-q^{16}(1+q+2q^2+q^3+q^4+q^5+q^6).
\end{align*}


We note that some relations between the Fibonacci numbers (resp. the perfect crystals) and
the Rogers-Ramanujan identities are known~\cite{An1,Cig} (resp. ~\cite{DL,Pri,TW} and references therein).
We also note that, after submission to arXiv of the first version of this paper,
a different proof of Theorem \ref{mainthm} was obtained~\cite{JU}.

It would be interesting to unify the length function in ~\cite{Ts1} and the $H$-length (and its modification)
as well as defining other length functions
depending on one's preference on explicit realizations (e.g., see a list of realizations in ~\cite{Kas}).

\hspace{0mm}

\noindent{\bf Organization of the paper.} 
In \S\ref{hlen}, we define the $H$-length for a Kyoto path realization, and prove Proposition \ref{mainres3}.
In \S\ref{exaA}, we apply it to a particular perfect crystal, and prove Theorem \ref{mainthm}.



\section{The $H$-length}\label{hlen}
In this section, $A$ is an affine Dynkin diagram, whose
vertices form a set $I$.
The fundamental null root is
given by $\delta=\sum_{i\in I}a_i\alpha_i$, where $a_i$
is the label at $i$ (see \cite{Kac}).

Let $\PB$ be perfect crystal of level $\ell$ (see ~\cite[Definition 1.1.1]{KKMMNN})
with an energy function $H:\PB\times\PB\to\mathbb{Z}$ (see ~\cite[\S4.1]{KKMMNN2}).
For a level $\ell$ dominant integral weight $\lambda=\sum_{i\in I}k_i\Lambda_i$ (i.e., $\sum_{i\in I}a^{\vee}_ik_i=\ell$, where $a^{\vee}_i$
is the colabel at $i$),
we have the ground-state path $\boldsymbol{g}=\cdots\otimes g_2\otimes g_1$ by
the condition
\begin{align*}
  \varphi(g_1)=\CL(\lambda), \textrm{ and }
  \varphi(g_{k+1})=\varepsilon(g_{k}) \textrm{ for } k\geq 1,
\end{align*}
where $\CL(\lambda)=\sum_{i\in I}k_i\overline{\Lambda_i}$ and $\varphi(b)=\sum_{i\in I}\varphi_i(b)\overline{\Lambda_i}$, $\varepsilon(b)=\sum_{i\in I}\varepsilon_i(b)\overline{\Lambda_i}$ for $b\in\PB$ in $P_{\CL}$.
For a positive integer $d$, 
we say that $\boldsymbol{g}$ is $d$-periodic 
if we have $g_{k+d}=g_k$ for $k\geq 1$.
One can define an $A$-crystal structure on the set
\begin{align*}
  \LP(\lambda) = \{\cdots\otimes b_2\otimes b_1\in\PB^{\otimes\infty}\mid 
  \textrm{$b_k\ne g_k$ holds only for finitely many $k$}
  \}
\end{align*}
of $\lambda$-paths so that we have an $A$-crystal isomorphism $\LP(\lambda)\cong B(\lambda)$~\cite[Proposition 4.6.4]{KKMMNN2}.
We define the $H$-length $\ell_H(\boldsymbol{b})$ of a $\lambda$-path $\boldsymbol{b}=\cdots\otimes b_2\otimes b_1$ by
\begin{align}
\ell_H(\boldsymbol{b})=\sum_{k\geq 1}(H(b_{k+1},b_k)-H(g_{k+1},g_k)).
\label{defH}
\end{align}

For a linear combination $y=\sum_{i\in I}y_i\alpha_i$, we define $\HT(y)=\sum_{i\in I}y_i$.

\begin{Lem}\label{fg}
Assume that the ground-state path $\boldsymbol{g}$ is $d$-periodic. 
There exist functions $f,g:\PB^d\times\PB^d\to\mathbb{Z}$ such that
\begin{align*}
\ell_H(\boldsymbol{b}\boldsymbol{p}) = \ell_H(\boldsymbol{b}) + f(\boldsymbol{q},\boldsymbol{p}),\quad
|\boldsymbol{b}\boldsymbol{p}| = |\boldsymbol{b}| + d\HT(\delta)\ell_H(\boldsymbol{b}) + g(\boldsymbol{q},\boldsymbol{p})
\end{align*}
for $\boldsymbol{p}=(p_d,\dots,p_1),\boldsymbol{q}=(q_d,\dots,q_1)\in\PB^d$ and $\boldsymbol{b}=\cdots\otimes b_2\otimes b_1\in\LP_{\boldsymbol{q}}(\lambda)$.
Here, $\boldsymbol{b}\boldsymbol{p}$ stands for the concatenation $\cdots\otimes b_2\otimes b_1\otimes p_d\otimes\cdots\otimes p_1$, and
\begin{align*}
  \LP_{\boldsymbol{q}}(\lambda) =\{\cdots\otimes b_2\otimes b_1\in\LP(\lambda)\mid\textrm{$b_k= q_k$ for $1\leq k\leq d$}\}
\end{align*}
is the set of $\lambda$-paths which begin with $\boldsymbol{q}$.
\end{Lem}

\begin{proof}
Let $p_{d+1}=q_1$. We show 
\begin{equation}
  \begin{split}
  f(\boldsymbol{q},\boldsymbol{p}) &=
  \sum_{k=1}^{d}(H(p_{k+1},p_k)-H(g_{k+1},g_k)),\\
  g(\boldsymbol{q},\boldsymbol{p}) &=
  \HT(\delta)\sum_{k=1}^{d}k(H(p_{k+1},p_k)-H(g_{k+1},g_k))
  -\sum_{k=1}^{d}\HT(\AF(\WT(p_k)-\WT(g_k))).
\end{split}
  \label{fg2}
\end{equation}
For $f$, the equality is obvious. 
For $g$, it follows from a formula~\cite[pp.503]{KKMMNN}
\begin{align*}
\WT(\boldsymbol{c})=\lambda+\sum_{k\geq 1}\AF(\WT(c_k)-\WT(g_k))-\delta\sum_{k\geq 1}k(H(c_{k+1},c_k)-H(g_{k+1},g_k))
\end{align*}
for $\boldsymbol{c}=\cdots\otimes c_2\otimes c_1\in \LP(\lambda)$, and
$|\boldsymbol{c}|=\HT(\lambda-\WT(\boldsymbol{c}))$. 
We remark that each of $f(\boldsymbol{q},\boldsymbol{p})$ and $g(\boldsymbol{q},\boldsymbol{p})$ depends only on $q_1$ and $\boldsymbol{p}$.
\end{proof}

\begin{Prop}\label{mainres3}
Assume that the ground-state path $\boldsymbol{g}$ is $d$-periodic. 
For a divisor $D$ of $d\HT(\delta)$ and a function $h:\PB^d\to\mathbb{Z}$, we define a function $\ell:\LP(\lambda)\to\mathbb{Z}$ by
\begin{align*}
  \ell(\boldsymbol{b})=D\ell_H(\boldsymbol{b})+h(\boldsymbol{q})
\end{align*}
for $\boldsymbol{b}\in \LP_{\boldsymbol{q}}(\lambda)$, where $\boldsymbol{q}\in\PB^d$.
The generating function
\begin{align*}
  F(x,q)=\sum_{\boldsymbol{b}\in\LP(\lambda)}x^{\ell(\boldsymbol{b})}q^{|\boldsymbol{b}|}
\end{align*}
satisfies a non-trivial $q$-difference equation. 
\end{Prop}

\begin{proof}
We define functions $\EF,\GI:\PB^d\times\PB^d\to\mathbb{Z}$ by
\begin{align}
  \EF(\boldsymbol{q},\boldsymbol{p}) =
  Df(\boldsymbol{q},\boldsymbol{p}) +h(\boldsymbol{p})-h(\boldsymbol{q}),\quad
  \GI(\boldsymbol{q},\boldsymbol{p}) =
  g(\boldsymbol{q},\boldsymbol{p}) - \frac{d\HT(\delta)}{D}h(\boldsymbol{q}).
\label{ffgg}
\end{align}
For $\boldsymbol{p},\boldsymbol{q}\in\PB^d$ and $\boldsymbol{b}\in\LP_{\boldsymbol{q}}(\lambda)$, it is easy to see, by Lemma \ref{fg}, that we have 
\begin{align}
\ell(\boldsymbol{b}\boldsymbol{p}) = \ell(\boldsymbol{b}) + \EF(\boldsymbol{q},\boldsymbol{p}),\quad
|\boldsymbol{b}\boldsymbol{p}| = |\boldsymbol{b}| + \frac{d\HT(\delta)}{D}\ell(\boldsymbol{b}) + \GI(\boldsymbol{q},\boldsymbol{p}).
\label{tildever}
\end{align}
We denote the generating function of $\LP_{\boldsymbol{p}}(\lambda)$
by $F_{\boldsymbol{p}}(x,q)$.
The formula \eqref{tildever}
implies
\begin{align*}
F_{\boldsymbol{p}}(x,q)
= \sum_{\boldsymbol{q}\in\PB^d}\sum_{\boldsymbol{b}\in\LP_{\boldsymbol{q}}(\lambda)}x^{\ell(\boldsymbol{b}\boldsymbol{p})}q^{|\boldsymbol{b}\boldsymbol{p}|}
= \sum_{\boldsymbol{q}\in\PB^d}x^{\EF(\boldsymbol{q},\boldsymbol{p})}q^{\GI(\boldsymbol{q},\boldsymbol{p})}F_{\boldsymbol{q}}(xq^{d\HT(\delta)/D},q).
\end{align*}
Because $(F_{\boldsymbol{p}}(x,q))_{\boldsymbol{p}\in\PB^d}$ satisfies a (non-trivial) simultaneous $q$-difference equation
\begin{align}
(F_{\boldsymbol{p}}(x,q))_{\boldsymbol{p}\in\PB^d}=
M\cdot
(F_{\boldsymbol{q}}(xq^{d\HT(\delta)/D},q))_{\boldsymbol{q}\in\PB^d},
\label{simu}
\end{align}
where $M=(x^{\EF(\boldsymbol{q},\boldsymbol{p})}q^{\GI(\boldsymbol{q},\boldsymbol{p})})_{\boldsymbol{p},\boldsymbol{q}\in\PB^d}$,
each $F_{\boldsymbol{p}}(x,q)$
satisfies a $q$-difference equation, which is obtained by the Murray-Miller algorithm (see ~\cite[Appendix B]{TT}).
Thus, the sum $F(x,q)=\sum_{\boldsymbol{p}\in\PB^d}F_{\boldsymbol{p}}(x,q)$ satisfies a $q$-difference equation (see ~\cite{GL,Kau,KK}).
\end{proof}

\section{A proof of Theorem \ref{mainthm}}\label{exaA}
We apply Proposition \ref{mainres3} to $A^{(1)}_1$ Kirillov-Reshetikhin perfect crystal $B^{1,3}$, whose crystal
graph is depicted as
\begin{align*}
\BI{0} \twoarrows{\boldsymbol{\longrightarrow}}{\longleftarrow}
\BI{1} \twoarrows{\boldsymbol{\longrightarrow}}{\longleftarrow}
\BI{2} \twoarrows{\boldsymbol{\longrightarrow}}{\longleftarrow}
\BI{3},
\end{align*}
where a thick (resp. thin) arrow is an $1$-arrow (resp. $0$-arrow).
We have $\delta=\alpha_0+\alpha_1$, and we may take $H(\BI{a},\BI{b})=\max(a-3,-b)$.
For $i=2$, we take $\lambda=3\Lambda_0$. The ground-state path is given by
$\boldsymbol{g}(=\cdots\otimes g_2\otimes g_1)=\cdots\otimes \BI{0}\otimes \BI{3}$,
which is $2$-periodic.

For $\boldsymbol{b}\in\LP(\lambda)$, we define
\begin{align*}
  \ell(\boldsymbol{b})=2\ell_H(\boldsymbol{b})-(3-b_1).
\end{align*}
It is not difficult to see that $\ell(\boldsymbol{b})$
is non-negative by a case-by-case analysis of
\begin{align*}
  (H(r,q)-H(3,0))+(H(q,p)-H(0,3)),
\end{align*}
which takes values in $\{0,1,2,3\}$ for $0\leq p,q,r\leq 3$.

As an instantiation of \eqref{fg2},
we have
\begin{align*}
  f((q_2,q_1),(p_2,p_1)) &= (H(q_1,p_2)-H(g_3,g_2))+(H(p_2,p_1)-H(g_2,g_1)),\\
  g((q_2,q_1),(p_2,p_1)) &= (p_2-g_2)+(p_1-g_1)+2((H(p_2,p_1)-H(g_2,g_1))+2(H(q_1,p_2)-H(g_3,g_2)))
\end{align*}
because of $\AF(\WT(a))=(2a-3)(\Lambda_0-\Lambda_1)$
and $\alpha_1=-2\Lambda_0+2\Lambda_1$,
which imply $\AF(\WT(a)-\WT(b))=-(a-b)\alpha_1$, where $0\leq a,b\leq 3$.

By \eqref{ffgg},
the $16\times 16$ matrix $M$ in \eqref{simu} is given as 
\begin{align*}
\begin{pmatrix}
\REP{x^6q^9} & \REP{x^5q^7} & \REP{x^4q^5} & \REP{x^3q^3} \\
\REP{x^4q^6} & \REP{x^3q^4} & \REP{x^2q^2} & \REP{x^3q^4} \\
\REP{x^2q^3} & \REP{xq} & \REP{x^2q^3} & \REP{x^3q^5} \\
\REP{1} & \REP{xq^2} & \REP{x^2q^4} & \REP{x^3q^6}  \\
\REP{x^5q^8} & \REP{x^4q^6} & \REP{x^3q^4} & \REP{x^2q^2}  \\ 
\REP{x^3q^5} & \REP{x^2q^3} & \REP{xq} & \REP{x^2q^3}   \\
\REP{xq^2} & \REP{1} & \REP{xq^2} & \REP{x^2q^4}  \\
\REP{xq} & \REP{x^2q^3} & \REP{x^3q^5} & \REP{x^4q^7} \\ 
\REP{x^4q^7} & \REP{x^3q^5} & \REP{x^2q^3} & \REP{xq}  \\
\REP{x^2q^4} & \REP{xq^2} & \REP{1} & \REP{xq^2}  \\
\REP{x^2q^3} & \REP{xq} & \REP{x^2q^3} & \REP{x^3q^5} \\ 
\REP{x^2q^2} & \REP{x^3q^4} & \REP{x^4q^6} & \REP{x^5q^8}  \\ 
\REP{x^3q^6} & \REP{x^2q^4} & \REP{xq^2} & \REP{1}  \\
\REP{x^3q^5} & \REP{x^2q^3} & \REP{xq} & \REP{x^2q^3} \\ 
\REP{x^3q^4} & \REP{x^2q^2} & \REP{x^3q^4} & \REP{x^4q^6} \\
\REP{x^3q^3} & \REP{x^4q^5} & \REP{x^5q^7} & \REP{x^6q^9} 
\end{pmatrix},
\end{align*}
where $\REP{z}$ stands for the four repetitions ``$z$ $z$ $z$ $z$'' of $z$, and $(\BI{a}\otimes\BI{b}=)(\BI{a},\BI{b})\in\PB^2$ corresponds to the index $1+4b+a$ for $0\leq a,b\leq 3$.

By these data and by computer calculation using the methods mentioned
in the proof of Proposition \ref{mainres3}, we get
\begin{align}
qF(x,q) 
= (1+xq)(1+q-xq+x^2q^3)F(xq,q)
-(1+xq^2)(1-x^2q^2)F(xq^{2},q).
\label{qdif}
\end{align}

A standard back-and-forth calculation proves Theorem \ref{mainthm} for $i=2$.
In fact, for $K(x,q)=\sum_{n\in\mathbb{Z}}k_n(q)x^n=F(x,q)/(-xq;q)_{\infty}$, we have
\begin{align*}
qK(x,q) 
= (1+q-xq+x^2q^3)K(xq,q)
-(1-xq)K(xq^{2},q).
\end{align*}
This is equivalent to the condition that, for $n\in\mathbb{Z}$, we have
\begin{align*}
qk_n=(q^n+q^{n+1})k_n-q^n k_{n-1}+q^{n+1}k_{n-2} - q^{2n}k_n + q^{2n-1}k_{n-1}.
\end{align*}
By putting $k_n=b_n^{(2)}/(q;q)_n$ for $n\geq 0$ (and $b^{(2)}_n=0$ for $n<0$), we get the recurrence relation for $b_n^{(2)}$ in Theorem \ref{mainthm} for $i=2$.

It is not difficult to prove that, for $\boldsymbol{c}=\cdots c_2\otimes c_1\in \LP(3\Lambda_0)$,
the condition $\ell(\boldsymbol{c})=1$ (resp. $\ell(\boldsymbol{c})=0$) is equivalent to the condition that there exists a positive integer $N$ such that $c_1=2$, $c_2=1$, $c_3=2$, $c_4=1$, ... and $c_{m}=g_m$ for $m>N$ (and then we have $|\boldsymbol{c}|=N$) (resp. $\boldsymbol{c}=\boldsymbol{g}$).
This implies $b_1^{(2)}=0$ (resp. $b_0^{(2)}=1$).

As in ~\cite[\S4]{Ts1}, we have $K(1,q)=1/(q^2,q^3;q^5)_{\infty}$, which is equal to $\sum_{n\geq 0}b_n^{(2)}/(q;q)_n$.
This completes a proof for $i=2$.
We omit a proof for $i=1$ because it is similar (take $\lambda=2\Lambda_0+\Lambda_1$
and define $\ell(\boldsymbol{b})=2\ell_H(\boldsymbol{b})-(2-b_1)$ for $\boldsymbol{b}\in\LP(\lambda)$).



\hspace{0mm}

\noindent{\bf Acknowledgements.}
The author was supported by the Research Institute for Mathematical
Sciences, an International Joint Usage/Research Center located in Kyoto
University, the TSUBAME3.0 supercomputer at Tokyo Institute of Technology,
JSPS Kakenhi Grant 20K03506, 23K03051, Inamori Foundation, JST CREST Grant Number JPMJCR2113, Japan and Leading Initiative for Excellent Young Researchers, MEXT, Japan.

\end{document}